\documentclass[11pt]{amsart}

\usepackage{amssymb}
\usepackage{bm}
\usepackage[centertags]{amsmath}
\usepackage{amsfonts}
\usepackage{amsthm}
\usepackage{mathrsfs}
\usepackage[usenames]{xcolor}

\theoremstyle{definition}
\newtheorem{thm}{Theorem}[section]
\newtheorem{dfn}[thm]{Definition}
\newtheorem{prp}[thm]{Proposition}
\newtheorem{lem}[thm]{Lemma}

\newtheorem{rmk}[thm]{Remark}

\newtheorem*{thm*}{Theorem}

\newtheorem*{cor*}{Corollary}

\newtheorem*{prp*}{Proposition}

\newcommand{\Ga}{\Gamma}
\newcommand{\ga}{\gamma}
\newcommand{\De}{\Delta}
\newcommand{\inn}{\in\mathbb{N}}
\newcommand{\al}{\alpha}
\newcommand{\de}{\delta}
\newcommand{\la}{\lambda}

\newcommand{\e}{\varepsilon}
\newcommand{\N}{\mathbb{N}}
\newcommand{\R}{\mathbb{R}}
\newcommand{\X}{\mathfrak{X}_0}

\DeclareMathOperator{\supp}{supp}

\DeclareMathOperator{\ran}{ran}

\DeclareMathOperator{\sgn}{sgn}

\DeclareMathOperator{\ra}{rank}

\DeclareMathOperator{\ag}{age}

\long\def\symbolfootnote[#1]#2{\begingroup%
\def\thefootnote{\fnsymbol{footnote}}\footnote[#1]{#2}\endgroup}

\allowdisplaybreaks

\begin{document}

\title[On the structure of separable $\mathcal{L}_\infty$-spaces]{On the structure of separable $\mathcal{L}_\infty$-spaces}

\author[S. A. Argyros \and I. Gasparis \and P. Motakis]{Spiros A. Argyros \and Ioannis Gasparis \and Pavlos Motakis}
\address{National Technical University of Athens, Faculty of Applied Sciences,
Department of Mathematics, Zografou Campus, 157 80, Athens, Greece}
\email{sargyros@math.ntua.gr}
\email{ioagaspa@math.ntua.gr}
\email{pmotakis@central.ntua.gr}

\begin{abstract}
Based on a construction method introduced by J. Bourgain and F. Delbaen, we give a general definition of a Bourgain-Delbaen space and prove that every infinite dimensional separable $\mathcal{L}_\infty$-space is isomorphic to such a space. Furthermore, we provide an example of a $\mathcal{L}_\infty$ and asymptotic $c_0$ space not containing $c_0$.
\end{abstract}

\maketitle

\symbolfootnote[0]{\textit{2010 Mathematics Subject Classification:} Primary 46B03, 46B06,  46B07}

\symbolfootnote[0]{\textit{Key words:} Bourgain-Delbaen spaces, separable $\mathcal{L}_\infty$-spaces, isomorphic $\ell_1$-preduals, asymptotic $c_0$ spaces.}
\symbolfootnote[0]{The authors would like to acknowledge the support of program API$\Sigma$TEIA-1082.}

\section{Introduction}
In the late 1960's J. Lindenstrauss and A. Pe\l czy\' nski introduced the class of $\mathcal{L}_\infty$-spaces, which naturally extends the class of $L_\infty$-spaces (\cite{LP}). Whether such spaces always contain a copy of $c_0$ remained a long standing open problem which was solved in the negative direction by J. Bourgain and F. Delbaen in \cite{BD}. In particular, they introduced a method for constructing $\mathcal{L}_\infty$-spaces and any example constructed with its use has been customarily called a Bourgain-Delbaen space. This method proved to be very fruitful, as it has been used extensively to construct a wide variety of $\mathcal{L}_\infty$-spaces, including the first example of a Banach space satisfying the ``scalar plus compact'' property (\cite{AH}), as well as to obtain other structural results in the geometry of Banach spaces (\cite{FOS}, \cite{7a}).

The main framework of the Bourgain-Delbaen method concerns the construction of an increasing sequence $(Y_n)_n$ of finite dimensional subspaces of a $\ell_\infty(\Ga)$ space, with $\Ga$ countably infinite, each one uniformly isomorphic to some $\ell_\infty^n$. This is achieved by carefully defining a sequence of extension operators $(i_n)_n$, each one defined on $\ell_\infty(\Ga_n)$ with $\Ga_n$ an appropriate finite subset of $\Ga$, and taking $Y_n$ to be the image of $i_n$. If the sequence $(i_n)_n$ satisfies a certain compatibility property, then the closure of the union of the $Y_n$, $n\inn$, is a $\mathcal{L}_\infty$-space whose properties depend on the definition of the aforementioned extension operators.

Based on this method, we give a broad definition of a Bourgain-Delbaen space. We include a brief study of the basic properties of such spaces and using techniques rooted in the early theory of $\mathcal{L}_\infty$-spaces (i.e. \cite{LP}, \cite{P}, \cite{LR}, \cite{JRZ}, \cite{LS}, \cite{S}) we prove that every separable infinite dimensional $\mathcal{L}_\infty$-space is isomorphic to such a space. We use this result to deduce that every separable infinite dimensional $\mathcal{L}_\infty$-space $X$ has an infinite dimensional $\mathcal{L}_\infty$-subspace $Z$, so that the quotient $X/Z$ is isomorphic to $c_0$.

In the final section of this paper we provide an example of an asymptotic $c_0$ isomorphic $\ell_1$-predual (i.e. a space whose dual is isomorphic to $\ell_1$) $\X$ that does not contain $c_0$. This result in particular yields that the proximity of a  Banach space to $c_0$ in a local setting, in the sense of being $\mathcal{L}_\infty,$ as well as an asymptotic setting, in the sense of being asymptotic $c_0$, does not imply proximity to $c_0$ in an infinite dimensional level. We think of this example as a step towards the solution of a problem in \cite[Question IV.2]{GL}, namely whether every isomorphic $\ell_1$-predual satisfying Pe\l czy\' nski's property-(u) is isomorphic to $c_0$. The question of the existence of an asymptotic $c_0$ $\mathcal{L}_\infty$-space not containing $c_0$ was asked by B. Sari.

\section{Bourgain-Delbaen spaces}
In the present section we give a general definition of the spaces that shall be called Bourgain-Delbaen spaces and prove some of their basic properties. It turns out that every infinite dimensional separable $\mathcal{L}_\infty$-space is isomorphic to a Bourgain-Delbaen space. We recall the definition of a $\mathcal{L}_\infty$-space, which was introduced in \cite[Definition 3.1]{LP}.
\begin{dfn}\label{script l p}
A Banach space $X$ is called a $\mathcal{L}_{\infty,C}$-space, for some $C\geqslant 1$, if for every finite dimensional subspace $F$ of $X$ there exists a finite dimensional subspace $G$ of $X$, containing $F$, which is $C$-isomorphic to $\ell_\infty^{n}$, where $n = \dim G$. A Banach space $X$ will be called a $\mathcal{L}_\infty$-space, if it is a $\mathcal{L}_{\infty,C}$-space for some $C\geqslant 1$.
\end{dfn}

\begin{rmk}\label{separable script l p}
If $X$ is an infinite dimensional separable Banach space then it is  well known, and not difficult to prove, that $X$ is a $\mathcal{L}_\infty$-space if and only if there exists a constant $C$ and an increasing sequence $(Y_n)_n$ of finite dimensional subspaces of $X$, whose union is dense in $X$, such that $Y_n$ is $C$-isomorphic to $\ell_\infty^{k_n}$, where $k_n = \dim Y_n$, for all $n\inn$.
\end{rmk}

\subsection{The definition of a Bourgain-Delbaen space.}\label{general definition of a BD space}
We give a broad definition of what kind of spaces we will refer to as Bourgain-Delbaen spaces.

\begin{dfn}
Let $\Gamma_1$, $\Ga$ be non-empty sets with $\Ga_1 \subset \Ga$. A linear operator $i: \ell_\infty(\Ga_1) \rightarrow \ell_\infty(\Ga)$ will be called an extension operator, if for every $x\in\ell_\infty(\Ga_1)$ and $\ga\in\Ga_1$ we have that $x(\ga) = i(x)(\ga)$.
\end{dfn}

\begin{dfn}\label{compatible}
Let $(\Ga_q)_{q=0}^\infty$ be a strictly increasing sequence of non-empty sets and $\Gamma = \cup_q\Ga_q$. A sequence of extension operators $(i_q)_{q=0}^\infty$, with $i_q: \ell_\infty(\Ga_q) \rightarrow \ell_\infty(\Ga)$ for all $q\inn\cup\{0\}$, will be called compatible if for every $p,q \inn\cup\{0\}$ with $p < q$ and $x\in\ell_\infty(\Ga_p)$, the following holds:
$$ i_p(x) = i_q(r_q(i_p(x))), $$
i.e. $i_p = i_q\circ r_q\circ i_p$, where $r_q:\ell_\infty(\Ga) \rightarrow \ell_\infty(\Ga_q)$ denotes the natural restriction operator.
\end{dfn}

\begin{dfn}\label{definition BD}
Let $(\Ga_q)_{q=0}^\infty$ be a strictly increasing sequence of non-empty finite sets, $\Ga = \cup_q\Ga_q$ and $(i_q)_{q=0}^\infty$, with $i_q: \ell_\infty(\Ga_q) \rightarrow \ell_\infty(\Ga)$ for all $q\inn\cup\{0\}$, be a sequence of compatible extension operators such that $C = \sup_q\|i_q\|$ is finite.
\begin{itemize}

\item[(i)] We define the sets $\De_0 = \Ga_0$ and $\De_q = \Ga_q\setminus\Ga_{q-1}$ for $q\inn\cup\{0\}$.

\item[(ii)] For every $\ga\in\Ga$ we define $d_\ga$, a vector in $\ell_\infty(\Ga)$, as follows: if $\ga\in\De_q$ for some $q\inn\cup\{0\}$, then $d_\ga = i_q(e_\ga)$.

\end{itemize}
The space $\mathfrak{X}_{(\Ga_q, i_q)_q} = \overline{\langle \{ d_\ga:\;\ga\in\Ga \} \rangle}$, i.e. the closed subspace of $\ell_{\infty}(\Ga)$ spanned by the vectors $(d_\ga)_{\ga\in\Ga}$, will be called a Bourgain-Delbaen space.
\end{dfn}

\begin{rmk}\label{they are isomorphisms}
Since the operators $i_q: \ell_\infty (\Ga_q) \rightarrow \ell_\infty(\Ga)$ are extension operators, it easily follows that $i_q$ is a $C$-isomorphism for all $q\inn\cup\{0\}$, where $C = \sup_q\|i_q\|$. In particular, the following hold:
\begin{itemize}

\item[(i)] if $Y_q = i_q[\ell_\infty(\Ga_q)]$, then $Y_q$ is $C$-isomorphic to $\ell_\infty(\Ga_q)$ and

\item[(ii)] the vectors $(d_\ga)_{\ga\in\De_q}$ are $C$-equivalent to the unit vector basis of $\ell_\infty(\De_q)$ for all $q\inn\cup\{0\}$.

\end{itemize}
\end{rmk}

\begin{rmk}\label{rephrase script l infinity}
The above Remark \ref{they are isomorphisms} and Proposition \ref{it is scrip L infinity} from Subsection \ref{properties bd subsec} provide an equivalent definition of a Bourgain-Delbaen space, namely the following.

Let $(\Ga_q)_q$ be a strictly increasing sequence of finite non-empty sets and $(Y_q)_q$ be an increasing sequence of subspaces of $\ell_\infty(\Ga)$, where $\Ga = \cup_q\Ga_q$. If there exists a constant $C>0$ such that for every $q\inn$, when restricted onto the subspace $Y_q$, the restriction operator  $r_q:Y_q\rightarrow \ell_\infty(\Ga_q)$ is an onto $C$-isomorphism, then the space $X = \overline{\cup_qY_q}$ is a Bourgain-Delbaen space.

Indeed, it is straightforward to check that the maps $i_q:\ell_\infty(\Ga_q)\rightarrow \ell_\infty(\Ga)$ with $i_q = r_q^{-1}:\ell_\infty(\Ga_q)\rightarrow Y_q\hookrightarrow \ell_\infty(\Ga)$ are uniformly bounded compatible extension operators and $\mathfrak{X}_{(\Ga_q,i_q)_q} = X$.
\end{rmk}

\subsection{Properties of a Bourgain-Delbaen space}\label{properties bd subsec}
We present some basic properties of Bourgain-Delbaen space, which can be deduced from Definition \ref{definition BD}.

\begin{prp}\label{it is an fdd}
Let $\mathfrak{X}_{(\Ga_q,i_q)_q}$ be a Bourgain-Delbaen space. For all $q\inn\cup\{0\}$ we denote by $M_q = \langle \{d_\ga: \ga\in\De_q\}\rangle$. Then $(M_q)_{q=0}^\infty$ forms a finite dimensional decomposition (FDD) for the space $\mathfrak{X}_{(\Ga_q,i_q)_q}$. More precisely, for every $p,q\inn\cup\{0\}$ with $p <q$ and $x_\ell\in M_\ell$ for $\ell = 0,1,\ldots,q$, the following holds:
\begin{equation}
\left\|\sum_{\ell = 0}^p x_\ell\right\| \leqslant C\left\|\sum_{\ell = 0}^q x_\ell\right\|
\end{equation}
where $C = \sup_q\|i_q\|$.
\end{prp}

\begin{proof}
Let $x_\ell = i_\ell(y_\ell)$ for some $y_\ell\in\langle\{e_\ga:\;\ga\in\De_\ell\}\rangle$, for $\ell = 0,1,\ldots,q$. Then, by the compatibility of the operators, we have that $x_\ell = i_p(r_p(i_\ell(y_\ell)))$ i.e. $x_\ell = i_p\circ r_p(x_\ell)$ for $\ell = 0,1,\ldots,p$ and hence:
\begin{equation}
\left\|\sum_{\ell = 0}^p x_\ell \right\| = \left\| i_p\circ r_p\left( \sum_{\ell = 0}^p x_\ell \right) \right\| \leqslant C \left\| r_p\left( \sum_{\ell = 0}^p x_\ell \right) \right\|.\label{it is an fdd1}
\end{equation}
On the other hand, once more by the extension property of the operators, we have that $r_p(x_\ell) = 0$ (which also yields that $i_p\circ r_p(x_\ell) = 0$) for $\ell = p+1,\ldots,q$ and therefore we obtain:
\begin{equation}
\left\|\sum_{\ell = 0}^q x_\ell \right\| \geqslant \left\|r_p\left(\sum_{\ell = 0}^q x_\ell\right) \right\| = \left\|r_p\left(\sum_{\ell = 0}^p x_\ell\right) \right\|.\label{it is an fdd2}
\end{equation}
The desired inequality immediately follows from \eqref{it is an fdd1} and \eqref{it is an fdd2}.
\end{proof}

\begin{rmk}\label{bd projections}
By the proposition above, for every interval $E = \{p,\ldots,q\}$ of $\mathbb{N}\cup\{0\}$ we can define the projection $P_E : \mathfrak{X}_{(\Ga_q,i_q)_q} \rightarrow M_p +\cdots + M_q$, associated to the FDD $(M_q)_q$ and the interval $E$. The above proof implies that
\begin{equation}\label{projection on initial interval}
P_{[0,q]}x = i_q\circ r_q(x)
\end{equation}
for all $q\inn\cup\{0\}$ and $x\in\mathfrak{X}_{(\Ga_q,i_q)_q}$ and hence also
\begin{equation}\label{projection on interval}
P_{(p,q]}x = i_q\circ r_q(x) - i_p\circ r_p(x)
\end{equation}
for all $p,q\inn\cup\{0\}$ with $p <q$ and $x\in\mathfrak{X}_{(\Ga_q,i_q)_q}$.
We shall call $P_E$ the Bourgain-Delbaen projection onto $E$. Note that $\|P_E\| \leqslant 2C$ for every interval $E$ of $\N\cup\{0\}$.
\end{rmk}

\begin{rmk}\label{it has a basis}
Proposition \ref{it is an fdd} in conjunction with Remark \ref{they are isomorphisms} (ii) yield that $\left((d_\ga)_{\ga\in\De_q}\right)_{q=0}^\infty$ is a Schauder basis of $\mathfrak{X}_{(\Ga_q,i_q)_q}$. Although in some cases it is more convenient to use the FDD $(M_q)_q$, in Section \ref{asymptotic c0} we shall indeed use this basis.
\end{rmk}


\begin{rmk}\label{block form}
Let $x$ be a finitely supported vector in $\mathfrak{X}_{(\Ga_q,i_q)_q}$ with $\ran x = E$. Using Remark \ref{bd projections} one can see that if $q = \max E$, there exists $y\in\ell_\infty(\Ga_q)$ with $\supp y\subset \cup_{p\in E}\De_p$ such that $x = i_q(y)$.
\end{rmk}

\begin{prp}\label{it is scrip L infinity}
Every Bourgain-Delbaen space $\mathfrak{X}_{(\Ga_q,i_q)_q}$ is a $\mathcal{L}_{\infty}$-space. More precisely, if $Y_q = i_q[\ell_\infty(\Ga_q)]$ for all $q\inn\cup\{0\}$, then $(Y_q)_q$ is a strictly increasing sequence of subspaces of $\mathfrak{X}_{(\Ga_q,i_q)_q}$, whose union is dense in the whole space and for every $q\inn\cup\{0\}$ $Y_q$ is $C$-isomorphic to $\ell_\infty(\Ga_q)$, where  $C = \sup_q\|i_q\|$.
\end{prp}

\begin{proof}
By Remark \ref{they are isomorphisms} (i) we have that $Y_q$ is $C$-isomorphic to $\ell_\infty(\Ga_q)$ for all $q\inn\cup\{0\}$. It remains to show that the sequence $(Y_q)_q$ is strictly increasing and its union is dense in $\mathfrak{X}_{(\Ga_q,i_q)_q}$. We will prove that $Y_q = M_0 + \cdots + M_q$, where $M_q = \langle\{d_\ga:\;\ga\in\De_q\}\rangle$ for all $q\inn\cup\{0\}$. Note that, in conjunction with Proposition \ref{it is an fdd}, the previous fact easily implies the desired result.

Note that for $p,q\inn\cup\{0\}$ with $p \leqslant q$ and $x\in M_p$ there is a $y\in\langle\{e_\ga:\;\ga\in\De_p\}\rangle$ with $x = i_p(y)$. Using the compatibility of the extension operators we obtain that $x = i_q(r_q(x))\in i_q(\ell_\infty(\Ga_q)) = Y_q$. We have hence concluded that $M_0 + \cdots + M_q\subset Y_q$. To conclude that the above inclusion cannot be proper, we will show that $\dim(M_0 + \cdots +M_q) = \dim Y_q$. Note that since, Proposition \ref{it is an fdd}, $(M_q)_{q=0}^\infty$ is an FDD, we have that $\dim(M_0 + \cdots +M_q) = \dim(M_0) + \cdots + \dim(M_q)$. Moreover, Remark \ref{they are isomorphisms} (ii) implies that $\dim(M_p) = \#\De_p$ for $p=0,\ldots,q$. The definition of the sets $\De_p$ yields that $\dim = (M_0 + \cdots +M_q) = \#\Ga_q$. Remark \ref{they are isomorphisms} (i) implies that $\dim Y_q = \#\Ga_q$ and therefore $\dim(M_0 + \cdots +M_q) = \dim Y_q$.
\end{proof}

\subsection{The functionals $(e_\ga^*)_{\ga\in\Ga}$}

Let $\mathfrak{X}_{(\Ga_q,i_q)_q}$ be a Bourgain-Delbaen space. For every $\ga\in\Ga$ we denote by $e_\ga^*:\mathfrak{X}_{(\Ga_q,i_q)_q}\rightarrow\mathbb{R}$ the evaluation functional on the $\ga$'th coordinate, defined on $\ell_\infty(\Ga)$ and then restricted to the subspace $\mathfrak{X}_{(\Ga_q,i_q)_q}$. Note that $\|e_\ga^*\| \leqslant 1$ for all $\ga\in\Ga$.

\begin{prp}\label{dual contains l1}
Let $\mathfrak{X}_{(\Ga_q,i_q)_q}$ be a Bourgain-Delbaen space. Then $(e_\ga^*)_{\ga\in\Ga}$ is $C$-equivalent to the unit vector basis of $\ell_1(\Ga)$, where $C = \sup_q\|i_q\|$.
\end{prp}

\begin{proof}
Let $A$ be a non-empty finite subset of $\Ga$ and $(\la_\ga)_{\ga\in A}$ be scalars. Choose $q\inn\cup\{0\}$ such that $A\subset \Ga_q$ and take the normalized vector $y = \sum_{\ga\in A}\sgn\la_\ga e_\ga$ in $\ell_{\infty}(\Ga_q)$. Note that $\|i_q(y)\| \leqslant C$. Since $i_q$ is an extension operator and $A\subset \Ga_q$ we obtain the following estimate:
\begin{eqnarray*}
C\sum_{\ga\in A}|\la_\ga| &\geqslant& C\left\|\sum_{\ga\in A}\la_\ga e_\ga^*\right\| \geqslant \sum_{\ga\in A}\la_\ga e_\ga^*(i_q(y)) = \sum_{\ga\in A}\la_\ga i_q(y)(\ga)\\
 &=& \sum_{\ga\in A}\la_\ga y(\ga) = \sum_{\ga\in A}\la_\ga \sgn \la_\ga = \sum_{\ga\in A} |\la_\ga|.
\end{eqnarray*}
i.e. $C^{-1}\sum|\la_\ga| \leqslant \|\sum\la_\ga e_\ga^*\| \leqslant \sum|\la_\ga|$, which is the desired result.	
\end{proof}

\begin{dfn}\label{extension functionals and biorthogonals}
Let $\mathfrak{X}_{(\Ga_q,i_q)_q}$ be a Bourgain-Delbaen space. For every $\ga\in\Ga$ we define two bounded linear functionals $c_\ga^*, d_\ga^*: \mathfrak{X}_{(\Ga_q,i_q)_q}\rightarrow \mathbb{R}$ as follows:
\begin{itemize}

\item[(i)] if $\ga\in\De_0$ then $c_\ga^* = 0$ and otherwise if $\ga\in\De_{q+1}$ for some $q\inn\cup\{0\}$, then $c_\ga^* = e_\ga^*\circ i_q\circ r_q$  and

\item[(ii)] $d_\ga^* = e_\ga^* - c_\ga^*$ for all $\ga\in\Ga$.

\end{itemize}
\end{dfn}

\begin{rmk}\label{coordinate composed with projection}
By Remark \ref{bd projections} obtain that if $q\inn\cup\{0\}$ and $\ga\in\De_{q+1}$, then $c_\ga^* = e_\ga^*\circ P_{[0,q]}$ and hence $d_\ga^* = e_\ga^*\circ P_{(q,\infty)}$. Also, using the extension property of the operators $i_q$, it easily follows from Remark \ref{bd projections} that for $p\inn\cup\{0\}$ and $\ga\in\Ga_p$ we have that $e_\ga^* = e_\ga^*\circ P_{[0,p]}$ which also implies that if $\ga\in\De_p$, then $d_\ga^* = e_\ga^*\circ P_{\{p\}}$. Moreover, for $\ga\in\De_0$, $d_\ga^* = e_\ga^*$.
\end{rmk}

\begin{lem}\label{extension functionals are in the span of evaluation functionals}
Let $\mathfrak{X}_{(\Ga_q,i_q)_q}$ be a Bourgain-Delbaen space. For all $q \inn\cup\{0\}$ and $\ga_0\in\De_{q+1}$ the functional $c_{\ga_0}^*$ is in the linear span of $\{e_\ga^*:\;\ga\in\Ga_q\}$.
\end{lem}

\begin{proof}
It suffices to show that $\cap_{\ga\in\Ga_q}\ker e_\ga^* \subset \ker c_{\ga_0}^*$ and to that end let $x\in\mathfrak{X}_{(\Ga_q,i_q)_q}$ with $e_\ga^*(x) = 0$ for all $\ga\in\Ga_q$. Recall that $x$ is also a vector in $\ell_\infty(\Ga)$ and in particular we have that $x(\ga) = e_\ga^*(x) = 0$ for all $\ga\in\Ga_q$ and hence, $r_q(x) = 0$. By the definition of $c_{\ga_0}^*$ it easily follows that $c_{\ga_0}^*(x) = 0$.
\end{proof}

\begin{prp}\label{facts about the dual}
Let $\mathfrak{X}_{(\Ga_q,i_q)_q}$ be a Bourgain-Delbaen space. The following hold.
\begin{itemize}

\item[(i)] The functionals $(d_\ga^*)_{\ga\in\Ga}$ are biorthogonal to the vectors $(d_\ga)_{\ga\in\Ga}$.

\item[(ii)] For all $q\inn\cup\{0\}$ we have that $\langle\{d_\ga^*:\;\ga\in\Ga_q\}\rangle = \langle\{e_\ga^*:\;\ga\in\Ga_q\}\rangle$. In particular, the closed linear span of the functionals $(d_\ga^*)_{\ga\in\Ga}$ is $C$-isomorphic to $\ell_1$, where $C = \sup_q\|i_q\|$.

\item[(iii)] If moreover the FDD $(M_q)_q$ is shrinking, then the closed linear span of the functionals $(d_\ga^*)_{\ga\in\Ga}$ is $\mathfrak{X}_{(\Ga_q,i_q)_q}^*$. In particular, $\mathfrak{X}_{(\Ga_q,i_q)_q}^*$ is $C$-isomorphic to $\ell_1$.

\end{itemize}
\end{prp}

\begin{proof}
The first statement easily follows from Remark \ref{coordinate composed with projection}, in particular the fact that for $q\inn\cup\{0\}$ and $\ga\in\De_q$ we have that $d_\ga^* = e_\ga^*\circ P_{\{q\}}$.


For the proof of (ii) observe that by Lemma \ref{extension functionals are in the span of evaluation functionals} and $d_\ga^* = e_\ga^* - c_\ga^*$, we have that $\langle\{d_\ga^*:\;\ga\in\Ga_q\}\rangle \subset \langle\{e_\ga^*:\;\ga\in\Ga_q\}\rangle$. Moreover (i) implies that $\dim \langle\{d_\ga^*:\;\ga\in\Ga_q\}\rangle = \#\Ga_q$ and hence the inclusion cannot be proper. The last part of statement (ii) follows from Proposition \ref{dual contains l1}.

To prove the last statement note that if  $(M_q)_q$ is shrinking, then so is the basis $\left((d_\ga)_{\ga\in\De_q}\right)_{q=0}^\infty$. From (ii) the desired result follows.
\end{proof}

\section{Separable $\mathcal{L}_\infty$-spaces are Bourgain-Delbaen spaces}\label{script L infinity are BD}
We combine some simple remarks concerning Bourgain-Delbaen spaces with results from \cite{JRZ}, \cite{LS} and \cite{S} to prove that whenever $X$ is an infinite dimensional separable $\mathcal{L}_\infty$-space, then $X$ is isomorphic to a Bourgain-Delbaen space.

\begin{lem}\label{good ell1 in measures}
Let $(\mu_i)_i$ be a bounded sequence in $\mathcal{M}[0,1]$. Then there exists a sequence $(t_i)_i$ of elements of $[0,1]$ such that if $\nu_i = \mu_i + \de_{t_i}$ for all $i\inn$ then the following hold:
\begin{itemize}

\item[(i)] the sequence $(\nu_i)_i$ is equivalent to the unit vector basis of $\ell_1(\mathbb{N})$ and

\item[(ii)] the space $Y = [(\nu_i)_i]$ is complemented in $\mathcal{M}[0,1]$.

\end{itemize}
\end{lem}

\begin{proof}
Find a sequence $(t_i)_i$ of elements of $[0,1]$ so that $\mu_j(\{t_i\}) = 0$ for all $i,j\inn$. Setting $\nu_i = \mu_i + \de_{t_i}$, it can be shown that $(\nu_i)_i$ is equivalent to the basis of $\ell_1$ and that the map $S\mu = \sum_{i=1}^\infty\mu(\{t_i\})\nu_i$ is a bounded projection onto the space $[(\nu_i)_i]$.
\end{proof}

\begin{lem}\label{good ell1 in dual}
Let $X$ be an infinite dimensional and separable $\mathcal{L}_\infty$-space. Then there exists a sequence $(x_i^*)_i$ in $X^*$ satisfying the following:
\begin{itemize}

\item[(i)] the sequence $(x_i^*)_i$ is equivalent to the unit vector basis of $\ell_1(\mathbb{N})$,

\item[(ii)] there exists a constant $\theta>0$ such that $\theta\|x\| \leqslant \sup_i|x_i^*(x)|$ for all $x\in X$ and

\item[(iii)] the space $Y = [(x_i^*)_i]$ is complemented in $X^*$.

\end{itemize}
\end{lem}

\begin{proof}
As it is shown in \cite{S}, the dual of $X$ is either isomorphic to $\ell_1(\mathbb{N})$ or to $\mathcal{M}[0,1]$ (see also \cite{LS}). In the first case, just choose a Schauder basis $(x_i^*)_i$ of $X^*$ which is equivalent to the unit vector basis of $\ell_1(\mathbb{N})$.

In the second case, let $T:X^*\rightarrow\mathcal{M}[0,1]$ be an onto isomorphism and choose a normalized sequence $(z_i^*)_i$ in $X^*$ such that $\|x\| = \sup_i|z_i^*(x)|$ for all $x\in X$. Fix $C > \|T^{-1}\|$ and apply Lemma \ref{good ell1 in measures} to the sequence $(\mu_i)_i$ with $\mu_i = CTz_i^*$ for all $i\inn$ to find a sequence $(t_i)_i$ in $[0,1]$ such that if $\nu_i = \mu_i + \de_{t_i}$ for all $i\inn$ then $(\nu_i)_i$ satisfies the conclusion of that lemma. Setting $x_i^* = T^{-1}\nu_i$ for all $i\inn$, it is not hard to check that, for $\theta = C - \|T^{-1}\|$, the sequence $(x_i^*)_i$ is the desired one.
\end{proof}

\begin{lem}\label{good bounded ap}
Let $X$ be a $\mathcal{L}_{\infty,\la}$-space, $Y$ be a subspace of $X^*$ and assume that there exists a constant $\theta>0$ so that for all $x\in X$, $\theta\|x\| \leqslant \sup\{|y^*(x)|:\;y^*\in B_{Y^*}\}.$
Then for every finite dimensional subspace $F$ of $X$ and every $\e>0$ there exists a finite rank operator $T:X\rightarrow X$ satisfying the following:
\begin{itemize}

\item[(i)] $\|Tx - x\| \leqslant \e\|x\|$ for all $x\in F$,

\item[(ii)] $\|T\| \leqslant \la/\theta$,

\item[(iii)] $T^*[X^*] \subset Y$.

\end{itemize}
In particular, if $X$ is separable then there exits a sequence of finite rank operators $T_n:X\rightarrow X$ with $T_nx\rightarrow x$ for all $x\in X$ and $T_n^*[X^*] \subset Y$ for all $n\inn$.
\end{lem}

\begin{proof}
Let $F$ be a finite dimensional subspace of $X$ and $\e>0$. By passing to a larger subspace, we may assume that $F = \langle\{y_i:\;i=1,\ldots,n\}\rangle$, where the map $A:\ell_\infty^n\rightarrow F$ with $Ae_i = y_i$ is invertible with $\|A\|\|A^{-1}\| \leqslant \la$. By the Hahn-Banach theorem there is a sequence $(y_i^*)_{i=1}^n$ in $X^*$, biorthogonal to $(y_i)_{i=1}^n$, such that $\|y_i^*\| \leqslant \|A^{-1}\|$ for $i=1,\ldots,n$. A separation theorem yields that $B_{X^*}\subset \theta^{-1}\overline{B_{Y}}^{w^*}$ and hence, we may choose $\tilde{y}_i^*\in \left(\theta^{-1}\|A^{-1}\|\right)B_Y$ such that $\left\|(\tilde{y}_i^* - y_i^*)|_{F}\right\| < \e/\|A\|$ for $i=1,\ldots,n$. Define $T:X\rightarrow X$ with $Tx = \sum_{i=1}^n\tilde{y}_i^*(x)y_i$. Some standard calculations yield that $T$ is the desired operator.
\end{proof}

The following terminology is from \cite{JRZ}. If $E_1, E_2$ are subspaces of a Banach space $X$ and $\e>0$, we say that $E_2$ is $\e$-close to $E_1$, if there is an invertible operator $T$ from $E_1$ onto $E_2$ with $\|Tx - x\| \leqslant \e\|x\|$ for all $x\in E_1$. Note that if $E_2$ is $\e$-close to $E_1$, then $E_1$ is $\e/(1-\e)$-close to $E_2$.

\begin{lem}\label{almost there}
Let $X$ be a Banach space and $(x_i^*)_{i=1}^\infty$ be a sequence in $X^*$ which is equivalent to the unit vector basis of $\ell_1(\mathbb{N})$. Assume moreover that there exists a sequence of bounded finite rank projections $(Q_n)_n$ satisfying the following:
\begin{itemize}

\item[(i)] $Q_n:X\rightarrow X$, $Q_n[X] \subset Q_{n+1}[X]$ for all $n\inn$ and $Q_nx\rightarrow x$ for all $x\in X$.

\item[(ii)] There is a strictly increasing sequence of natural numbers $(k_n)_n$ and $0<\e<1$ with $\e/(1-\e)<(\sup_n\|Q_n\|)^{-1}$ such that $Q^*_n[X^*]$ is $\e$-close to the space $\langle\{x_i^*:\;i=1,\ldots,k_n\}\rangle$ for all $n\inn$.

\end{itemize}
Then $X$ is isomorphic to a Bourgain-Delbaen space.
\end{lem}

\begin{proof}
Note that (i) implies that $\cup_n Q_n[X]$ is dense in $X$. Define a linear operator $U: X\rightarrow \ell_\infty(\mathbb{N})$ with $Ux = (x_i^*(x))_i$ for all $x\in X$, evidently $U$ is bounded. Set $\Ga_n = \{1,\ldots,k_n\}$ for all $n\inn$ and $Y_n = UQ_n[X]$. We shall prove that $U$ is an isomorphic embedding and that the restriction operators onto the first $k_n$ coordinates $r_{k_n}: Y_n\rightarrow \ell_\infty(\{1,\ldots,k_n\})$ are onto $C$-isomorphisms for all $n\inn$, for a uniform constant $C$. Given the aforementioned facts and Remark \ref{rephrase script l infinity}, the desired result follows easily.

Since $\cup_n Q_n[X]$ is dense in $X$, to show that $U$ is an isomorphic embedding, it is enough to find a uniform constant $c > 0$, such that $\|Ux\| \geqslant c\|x\|$ for every $x\in Q_n[X]$ and for every $n\inn$. To this end let $S:[(x_i^*)_i]\rightarrow \ell_1(\mathbb{N})$ be the onto isomorphism with $Sx_i^* = e_i$ for all $i\inn$. Let $n\inn$ and $x\in Q_n[X]$. The Hahn-Banach Theorem and the fact that $\langle\{x_i^*:\;i=1,\ldots,k_n\}\rangle$ is $\e/(1-\e)$-close to $Q^*_n[X^*]$, in conjunction with some tedious computations, yield that there exists an $i_0\in\{1,\ldots,k_n\}$ such that
\begin{equation}\label{ahmahgahd}
|x_{i_0}^*(x)| \geqslant (1-\e)\frac{1-\frac{\e}{1-\e}\|Q_n\|}{\|S\|\|Q_n\|}\|x\|.
\end{equation}
Setting $c = (1-\e)(1 - (\e/(1-\e))\sup_k\|Q_k\|)/(\|S\|\sup_k\|Q_k\|)$ we conclude $\|Ux\| \geqslant c\|x\|$.

It remains to show that there a exists a constant $C>0$ such that $r_{k_n}: Y_n\rightarrow \ell_\infty(\{1,\ldots,k_n\})$ is an onto $C$-isomorphism for every $n\inn$.
Let $n\inn$ and $z\in Y_n$. Set $x = U^{-1}z\in Q_n[X]$. Then by \eqref{ahmahgahd} there is an $i_0\in\{1,\ldots,k_n\}$ such that:
 $$\|z\| \geqslant \|r_{k_n}(z)\| \geqslant |x_{i_0}^*(x)| \geqslant c\|x\|\geqslant \frac{c}{\|U\|}\|z\|.$$
 Setting $C = \|U\|/c$ we conclude that $r_{k_n}|_{Y_n}$ is a $C$-isomorphic embedding. Finally, observe that $\dim Y_n = \dim Q_n[X] = \dim Q_n^*[X^*] = k_n$ and hence $r_{k_n}|_{Y_n}$ is also onto.
\end{proof}

The following lemma and its proof can be found in \cite[Lemma 4.3]{JRZ}.

\begin{lem}\label{Johnson Rosenthal Zippin}
Let $X$ be a separable Banach space and $Y$ be a separable subspace of $X^*$. Assume moreover that $(P_n)_n$, $(T_n)_n$ are bounded finite rank operators satisfying the following conditions:
\begin{itemize}

\item[(a)] $P_n:X\rightarrow X$, $P_n^*[X^*] \subset Y$ and $T_n: X^*\rightarrow Y$ for all $n\inn$.

\item[(b)] $P_nx\rightarrow x$ for all $x\in X$, $T_ny\rightarrow y$ for all $y\in Y$ and $\sup_n\|T_n\| < \infty$.

\item[(c)] The operators $(T_n)_n$ are projections.

\end{itemize}
If $E$, $F$ are finite dimensional subspaces of $X$, $X^*$ respectively and $0<\e<1$, then there exists a projection $Q:X\rightarrow X$ with finite dimensional range such that:
\begin{itemize}

\item[(i)] $Qe = e$ for all $e\in E$,

\item[(ii)] $Q^*f = f$ for all $f\in F$,

\item[(iii)] $Q^*[X^*] \subset Y$,

\item[(iv)] $\|Q\| \leqslant 2c + 2K + 4cK$ where $c = \sup_n\|P_n\|$ and $K = \sup_n\|T_n\|$,

\item[(v)] $Q^*[X^*]$ is $\e$-close to $T_n[X^*]$ for some integer $n$.

\end{itemize}
\end{lem}

\begin{thm}\label{separable script l infinities are bd spaces}
Every separable infinite dimensional $\mathcal{L}_\infty$-space is isomorphic to a Bourgain-Delbaen space.
\end{thm}

\begin{proof}
Let $X$ be a separable infinite dimensional $\mathcal{L}_\infty$-space.
By Lemma \ref{good ell1 in dual}, there exists a sequence $(x_i^*)_i$ in $X^*$ which is equivalent to the unit vector basis of $\ell_1(\mathbb{N})$, a constant $\theta>0$ such that $\theta\|x\| \leqslant \sup_i|x_i^*(x)|$ for all $x\in X$ and a bounded linear projection $P: X^*\rightarrow Y = [(x_i^*)_i]$. For $n\inn$ define $T_n:X^*\rightarrow Y$ with $T_n =  S_n\circ P$, where $S_n: Y\rightarrow Y$ denotes the basis projection of $(x_i^*)_i$ onto the first $n$ coordinates. Also apply Lemma \ref{good bounded ap} to find a sequence of finite rank operators $P_n:X\rightarrow X$ such that $P_nx\rightarrow x$ for all $x\in X$ and $P_n^*[X^*]\subset Y$ for all $n\inn$. Assumptions (a), (b) and (c) of Lemma \ref{Johnson Rosenthal Zippin} are evidently satisfied.

Choose $\e>0$ with $\e/(1-\e) < 1/(2c + 2K + 4cK)$ where $c = \sup_n\|P_n\|$ and $K = \sup_n\|T_n\|$. Recursively, setting $E_n = P_n[X] + Q_{n-1}[X]$, where $Q_0$ is the zero operator on $X$, and $F_n = \{0\}$, using Lemma \ref{Johnson Rosenthal Zippin}, choose $(Q_n)_n$ a sequence of bounded linear projections on $X$ such that $P_n[X]\subset Q_n[X] \subset Q_{n+1}[X]$,  $\|Q_n\| \leqslant 2c + 2K + 4cK$ for all $n\inn$ and there exist natural numbers $(k_n)_n$, such that $Q_n^*[X^*]$ is $\e$-close to $T_{k_n}[X^*] = \langle\{x_i^*:\;i=1,\ldots,k_n\}\rangle$ for all $n\inn$. Note that $\dim Q_n[X]\rightarrow \infty$ and hence, by passing to a subsequence we may assume that $(k_n)_n$ is strictly increasing. We conclude that the sequences $(x_i^*)_{i=1}^\infty$ and $(Q_n)_n$ witness the fact that the space $X$ satisfies the assumptions of Lemma \ref{almost there} and therefore it is isomorphic to a Bourgain-Delbaen space.
\end{proof}

\section{Separable infinite dimensional $\mathcal{L}_\infty$-spaces contain $\mathcal{L}_\infty$-subspaces of infinite co-dimension}
Using the main result of Section \ref{script L infinity are BD} we prove that every infinite dimensional $\mathcal{L}_\infty$-space $X$ contains an infinite dimensional $\mathcal{L}_\infty$-subspace $Z$ so that the quotient $X/Z$ is isomorphic to $c_0$.

\begin{lem}\label{conditions to contain co-infinite dimensional script L infinity}
Let $\mathfrak{X}_{(\Ga_q,i_q)_q}$ be a Bourgain-Delbaen space and assume that there exists a decreasing sequence of positive real numbers $(\e_q)_{q=1}^\infty$, with $2C^2\sum_q\e_q < 1$ where $C = \sup_q\|i_q\|$, so that for every $q\inn$ there exist $\ga_1^q, \ga_2^q\in\De_q$ with $\ga_1^q\neq\ga_2^q$ satisfying $\|e_{\ga_1^q}^*\circ i_p\circ r_p - e_{\ga_2^q}^*\circ i_p\circ r_p\| < \e_q$ for $p=0,1,\ldots,q-1$. Then $\mathfrak{X}_{(\Ga_q, i_q)_q}$ contains an infinite dimensional $\mathcal{L}_\infty$-subspace $Z$.
\end{lem}

\begin{proof}
Define $R_q = \{\ga_1^q, \ga_2^q\}$, $S_q = \cup_{p=1}^q R_p$ for $q\inn$ and $S = \cup_{q=1}^\infty S_q$. Define $N_0 = M_0 = \langle\{d_{\ga}:\;\ga\in\De_0\}\rangle$ and
$$N_q = \langle\{d_\ga:\;\ga\in\De_q\setminus R_q\}\cup\{d_{\ga_1^q} + d_{\ga_2^q}\}\rangle$$
for $q\inn$. Note that $N_q$ is a subspace of $M_q$ of co-dimension one for every $q\inn$. Define
\begin{equation*}
Z_q = \langle\{d_\ga:\;\ga\in\Ga_q\setminus S_q\}\cup\{d_{\ga_1^p} + d_{\ga_2^p}:\;p=1,\ldots,q\}\rangle
\end{equation*}
for all $q\inn$. Observe that $Z_q = N_0 + \cdots + N_q$ for all $q\inn$ and that $(Z_q)_q$ is an increasing sequence of finite dimensional spaces whose union is dense in the space
\begin{equation*}
Z = \overline{\langle\{d_\ga:\;\ga\in\Ga\setminus S\}\cup\{d_{\ga_1^q} + d_{\ga_2^q}:\;q\inn\}\rangle}.
\end{equation*}
We shall prove that $Z$ is the desired subspace. To begin, $Z$ is clearly infinite dimensional. More precisely, the spaces $(N_q)_q$ form an FDD for the space $Z$ with a projection constant $C$.

To conclude that $Z$ is indeed a $\mathcal{L}_\infty$-space, it suffices to show that for each $q\inn$ the space $Z_q$ is $\frac{1+\e}{1-\e}C$-isomorphic to $\ell_\infty^{n_q}$ where $\e = 2C^2\sum_{q=1}^\infty\e_q$  and $n_q = \dim Z_q = \#\Ga_q - q$. For $q\inn$ we define a subspace of $\ell_\infty(\Ga_q)$ as follows:
\begin{equation*}
W_q = \langle\{e_\ga:\;\ga\in\Ga_q\setminus S_q\}\cup\{e_{\ga_1^p} + e_{\ga_2^p}:\;p=1,\ldots,q\}\rangle.
\end{equation*}
Also define $\widetilde{W}_q = i_q[W_q]$. Observe that $W_q$ is isometric to $\ell_\infty^{n_q}$ and therefore, by Remark \ref{they are isomorphisms}, $\widetilde{W}_q$ is $C$-isomorphic to $\ell_\infty^{n_q}$. Hence, if we find a linear map $T_q: Z_q\rightarrow \widetilde{W}_q$ with $\|T_qx - x\| \leqslant \e\|x\|$ for all $x\in Z_q$ the proof will be complete.

Let us fix $q\inn$ and observe that if $x\in N_p$, for some $0\leqslant p\leqslant q$, then $x = i_p(y)$ where $y \in\langle\{e_\ga:\;\ga\in\De_q\setminus S_q\}\cup\{e_{\ga_1^q} + e_{\ga_2^q}\}\rangle$. For $p=0,\ldots,q$ we define $T_{q,p}:N_p\rightarrow W_q$ as follows:
\begin{equation*}
T_{q,p}(x)(\ga) = \left\{
\begin{array}{ll}
y(\ga)& \text{if } \ra(\ga)\leqslant p,\\
i_p(y)(\ga)& \text{if } \ra(\ga)>p\;\text{and}\;\ga\notin S_q,\\
i_p(y)(\ga_1^s)& \text{if } \ra(\ga)>p\;\text{and}\;\ga = \ga_1^s\;\text{for some}\;s\in(p,q],\\
i_p(y)(\ga_1^s)& \text{if } \ra(\ga)>p\;\text{and}\;\ga = \ga_2^s\;\text{for some}\;s\in(p,q].
\end{array}
\right.
\end{equation*}
Observe that the map is linear and well defined. Also observe that if $x(\ga)\neq T_{q,p}(x)(\ga)$, for some $\ga\in\Ga_q$, then necessarily there is an $s\in(p,q]$ so that $\ga = \ga_2^s$ and $T_{q,p}(x)(\ga) = i_p(y)(\ga_1^s) = e_{\ga_1^s}^*\circ i_p\circ r_p(x)$. Hence, for such a $\ga\in\Ga_q$:
\begin{equation*}
|x(\ga) - T_{q,p}(x)(\ga)| = |e_{\ga_2^s}^*\circ i_p\circ r_p(x) - e_{\ga_1^s}^*\circ i_p\circ r_p(x)| < \e_s\|x\| \leqslant \e_{p+1}\|x\|.
\end{equation*}
We conclude that $\|r_q(x) - T_{q,p}(x)\| \leqslant \e_{p+1}\|x\|$ (actually observe that if $p=q$ then $r_q(x) = T_{q,q}(x)$) and hence if we define $\widetilde{T}_{q,p}: N_p\rightarrow \widetilde{W}_q$ with $\widetilde{T}_{q,p} = i_q\circ T_{q,p}$ then
\begin{equation*}
\left\|x - \widetilde{T}_{q,p}(x)\right\| \leqslant C\e_{p+1}\|x\|
\end{equation*}
for all $x\in N_p$ and $p = 0,\ldots,q$. As we previously mentioned, $(N_p)_p$ is an FDD for $Z$ with a projection constant $C$, so we may consider the associated projections $Q_{\{p\}}:Z\rightarrow N_p$ for all $p$. Define $T_q: Z_q\rightarrow \widetilde{W}_q$ with $T_q = \sum_{p=0}^q\widetilde{T}_{q,p}\circ Q_{\{p\}}$. Some simple calculations using $\|Q_{\{p\}}\| \leqslant 2C$ for all $p$ yield that $T_q$ is the desired operator.
\end{proof}

We recall that it is known that separable $\mathcal{L}_\infty$-spaces have $c_0$ as a quotient \cite{P}, \cite{AB}.

\begin{lem}\label{quotient is c0}
Let $\mathfrak{X}_{(\Ga_q,i_q)_q}$ be a Bourgain-Delbaen space satisfying the assumptions of Lemma \ref{conditions to contain co-infinite dimensional script L infinity}. If $Z$ is the subspace constructed in that lemma then the quotient $\mathfrak{X}_{(\Ga_q,i_q)_q}/Z$ is isomorphic to $c_0(\mathbb{N})$.
\end{lem}

\begin{proof}
Following the notation of the proof of Lemma \ref{conditions to contain co-infinite dimensional script L infinity}, let $Z$ be the constructed subspace and set also $W$ be the closed linear span of the vectors $d_{\ga^q_1} - d_{\ga^q_2}$, $q\inn$. Let $Q:\mathfrak{X}_{(\Ga_q,i_q)_q}\rightarrow \mathfrak{X}_{(\Ga_q,i_q)_q}/Z$ denote the corresponding quotient map and for $q\inn$ define $y_q = Q(d_{\ga^q_1} - d_{\ga^q_2})$. It is not very difficult to prove that $(y_q)_q$ is a Schauder basis of $\mathfrak{X}_{(\Ga_q,i_q)_q}/Z$ with projection constant $C = \sup_q\|i_q\|$. Let $w_q^* = 1/2(d_{\ga_1^q}^* - d_{\ga_2^q}^*)$ for all $q\inn$ and $(y_q^*)_q\subset (\mathfrak{X}_{(\Ga_q,i_q)_q}/Z)^*$ be the biorthogonal sequence of $(y_q)_q$. It follows that the sequences $(w_q^*)_q$ and $(y_q^*)_q$ are isometrically equivalent and that $\|w_q^* - 2(e_{\ga^q_1}^* - e_{\ga^q_2}^*)\| \leqslant 2\e_q$ for all $q$. Proposition \ref{dual contains l1} yields that $(y_q^*)_q$ is equivalent to the unit vector basis of $\ell_1$, which indeed implies that $(y_q)_q$ is equivalent to the unit vector basis of $c_0(\mathbb{N})$.
\end{proof}

\begin{prp}\label{all BD's contain co-infinite dimensional script L infinity}
Every separable infinite dimensional $\mathcal{L}_\infty$-space $X$ contains an infinite dimensional $\mathcal{L}_\infty$-subspace $Z$ so that the quotient $X/Z$ is isomorphic to $c_0(\mathbb{N})$. In other words, every separable infinite dimensional $\mathcal{L}_\infty$-space $X$ is the twisted sum of a $\mathcal{L}_\infty$-space $Z$ and $c_0(\mathbb{N})$.
\end{prp}

\begin{proof}
By virtue of Theorem \ref{separable script l infinities are bd spaces}, we may assume that  $X$ is a Bourgain-Delbaen space $\mathfrak{X}_{(\Ga_q,i_q)_q}$. We start by observing that for every strictly increasing sequence $(q_s)_{s=0}^\infty$ of $\mathbb{N}\cup\{0\}$, the Bourgain-Delbaen spaces $\mathfrak{X}_{(\Ga_q,i_q)_q}$ and $\mathfrak{X}_{(\Ga_{q_s},i_{q_s})_s}$ are actually equal. This follows from Proposition \ref{it is scrip L infinity}, in particular the fact that $\cup_qY_q$ is dense in $\mathfrak{X}_{(\Ga_q,i_q)_q}$. It is therefore sufficient to find an appropriate increasing sequence $(q_s)_{s=0}^\infty$ so that the space $\mathfrak{X}_{(\Ga_{q_s},i_{q_s})_s}$ satisfies the assumptions of Lemma \ref{conditions to contain co-infinite dimensional script L infinity} (and hence also those of Lemma \ref{quotient is c0}).

Fix a decreasing sequence of positive real numbers $(\e_q)_{q=1}^\infty$, with $\sum_q\e_q < 1/(2C^2)$ where $C = \sup_q\|i_q\|$. A compactness argument yields that for every $\e>0$ and $q\inn\cup\{0\}$ there exist non-equal $\ga_1,\ga_2\in\Ga\setminus\Ga_q$ so that $\|e_{\ga_1}^*\circ i_p\circ r_p - e_{\ga_2}^*\circ i_p\circ r_p\| < \e$ for $p=0,\ldots,q$. Set $q_0 = 0$ and using the above, recursively choose a strictly increasing sequence $(q_s)_{s=0}^\infty$ so that for every $s\inn$ there are non-equal $\ga_1^s,\ga_2^s\in\Ga_{q_s}\setminus \Ga_{q_{s-1}}$ with $\|e_{\ga_s^1}^*\circ i_{q_t} \circ r_{q_t} - e_{\ga_s^2}^*\circ i_{q_t} \circ r_{q_t}\| \leqslant \e_s$ for $t = 0,\ldots, s-1$.
\end{proof}

\section{A $\mathcal{L}_\infty$ and asymptotic $c_0$ space not containing $c_0$}\label{asymptotic c0}
We employ the Bourgain-Delbaen method to define an isomorphic $\ell_1$-predual $\X$, which is asymptotic $c_0$ and does not contain a copy of $c_0$. We follow notation similar to that used in \cite{BD} and \cite{AH}.

\subsection{Definition of the space $\X$} We fix a natural number $N\geqslant 3$ and a constant $1<\theta < N/2$. Define $\De_0 = \{0\}$ and assume that we have defined the sets $\De_0,\De_1,\ldots,\De_{q}$. We set $\Ga_p = \cup_{i=0}^p\Ga_i$ and for each $\ga\in\Ga_q$ We denote by $\ra(\ga)$ the unique $p$ so that $\ga\in\De_p$. Assume that to each $\ga\in\Ga_q\setminus \Ga_0$ we have assigned a natural number in $\{1,\ldots,n\}$, called the age of $\ga$ and denoted by $\ag(\ga)$.

Assume moreover that we have defined extension functionals $(c_\ga^*)_{\ga\in\De_p}$ and extension operators $i_{p-1,p}:\ell_{\infty}(\Ga_{p-1})\rightarrow \ell_{\infty}(\Ga_{p})$ so that
\begin{equation}\label{extension to next in this particular case}
i_{p-1,p}(x)(\ga) = \left\{
\begin{array}{ll}
x(\gamma)& \text{if } \gamma\in\Gamma_{p-1},\\
c_\gamma^*(x)& \text{if } \gamma\in \Delta_{p}
\end{array}
\right.
\end{equation}
for $p=1,\ldots,q$. If $0\leqslant p <s\leqslant q$ we define $i_{p,s} = i_{s-1,s}\circ\cdots\circ i_{p,p+1}$ and we also denote by $i_{p,p}$ the identity operator on $\ell_{\infty}(\Ga_{p})$.

We define $\De_{q+1}$ to be the set of all tuples of one of the two forms described below:
\begin{subequations}
\begin{equation}\label{simple type tuple}
\left(q+1,n,(\e_i)_{i=1}^k, (E_i)_{i=1}^k, (\eta_i)_{i=1}^k\right)
\end{equation}
where $n\leqslant (\#\Ga_q)^2$, $1\leqslant k \leqslant n$, $\e_i\in \{-1,1\}$ for $i=1,\ldots,k$, $(E_i)_{i=1}^k$ is a sequence of successive intervals of $\{0,\ldots,q\}$ and $\eta_i\in \Ga_q$ with $\ra(\eta_i)\in E_i$ for $i=1,\ldots,k$, or
\begin{equation}\label{slightly less simple type tuple}
\left(q+1,\xi,n,(\e_i)_{i=1}^k, (E_i)_{i=1}^k, (\eta_i)_{i=1}^k\right)
\end{equation}
where $\xi\in\Ga_{q-1}\setminus\Ga_0$ with $\ag(\xi) < N$, $(\#\Ga_{\ra(\xi)})^2 \leqslant n \leqslant (\#\Ga_{q})^2$, $1\leqslant k\leqslant n$, $\e_i\in \{-1,1\}$ for $i=1,\ldots,k$, $(E_i)_{i=1}^k$ is a sequence of successive intervals of $\{\ra(\xi)+1,\ldots,q\}$ and $\eta_i\in \Ga_q$ with $\ra(\eta_i)\in E_i$ for $i=1,\ldots,k$.
\end{subequations}

\begin{subequations}
To each $\ga\in\De_{q+1}$ we assign an extension functional $c_\ga^*:\ell_{\infty}(\De_q)\rightarrow \R$. If $0\leqslant p\leqslant q$ we define the projection $P^q_{[0,p]}x = i_{p,q}\circ r_p(x)$ while if $0 \leqslant p\leqslant s\leqslant  q$ we define the projection $P^q_{(p,s]} = P^q_{[0,s]} - P^q_{[0,p]}$. If $\ga\in\De_{q+1}$ is of the form \eqref{simple type tuple}, we set $\ag(\ga) = 1$ and define the extension functional $c_\ga^*$ as follows:
\begin{equation}\label{simple type extension}
c_\ga^* = \frac{\theta}{N}\frac{1}{n}\sum_{i=1}^k\e_i e_{\eta_i}^*\circ P^q_{E_i}.
\end{equation}
If $\ga\in\De_{q+1}$ is of the form \eqref{slightly less simple type tuple}, we set $\ag(\ga) = \ag(\xi) + 1$ and define the extension functional $c_\ga^*:\ell_{\infty}(\De_q)\rightarrow \R$ as follows:
\begin{equation}\label{slightly less simple type extension}
c_\ga^* = e_\xi^* + \frac{\theta}{N}\frac{1}{n}\sum_{i=1}^k\e_i e_{\eta_i}^*\circ P^q_{E_i}.
\end{equation}
\end{subequations}

The inductive construction is complete. We set $\Ga = \cup_{q=1}^\infty\De_q$ and for each $q\in\N\cup\{0\}$ we define the extension operator $i_q:\ell_\infty(\Ga_q)\rightarrow\ell_\infty(\Ga)$ by the rule
\begin{equation*}
i_q(x)(\ga) = \left\{
\begin{array}{ll}
x(\gamma)& \text{if } \gamma\in\Gamma_q,\\
i_{q,p}(x)(\ga)& \text{if } \gamma\in \Delta_{p}\;\text{for some}\;p > q.
\end{array}
\right.
\end{equation*}
Standard arguments yield that $(i_q)_{q=0}^\infty$ is a compatible sequence of extension operators with $\sup_q\|i_q\| \leqslant N/(N-2\theta)$. We denote by $\X$ the resulting Bourgain-Delbaen space  $\mathfrak{X}_{(\Ga_q,i_q)_q}$. We shall use the notation from section \ref{general definition of a BD space}.

\begin{rmk}\label{bd norms in this particular case}
By Remark \ref{bd projections} we obtain that for every interval $E$ of the natural numbers, $\|P_E\| \leqslant 2N/(N-2\theta)$. It follows that for every $\ga\in\Ga$ and such $E$, $\|e_\ga^*\circ P_E\| \leqslant 2N/(N-2\theta)$, in particular $\|d_\ga^*\| \leqslant 2N/(N - 2\theta)$.
\end{rmk}

\begin{rmk}
We enumerate the set $\Ga = \cup_{q=0}^\infty\De_q$ in such a manner that the sets $\De_q$ correspond to successive intervals of $\N$. If we denote this enumeration by $\Ga = \{\ga_i:\;\inn\}$, according to Remark \ref{it has a basis}, $(d_{\ga_i})_i$ forms a Schauder basis for $\X$. It is with respect to this basis that we show that the space $\X$ is asymptotic $c_0$. However, when we write $P_E$ we shall mean the Bourgain-Delbaen projection onto $E$ as defined in Remark \ref{bd projections}.
\end{rmk}

Arguing as in \cite[Proposition 4.5]{AH}, each $e_\ga^*$ admits an analysis.
\begin{prp}\label{analysis}
Let $\ga\in\Ga$ with $\ra(\ga) > 0$. The functional $e_\ga^*$ admits a unique analysis of the following form:
\begin{equation}\label{analyse this}
e_\ga^* = \sum_{r = 1}^a d_{\xi_r}^* + \frac{\theta}{N}\sum_{r=1}^a \frac{1}{n_r} \sum_{i=1}^{k_r}\e_{r,i}e_{\eta_{r,i}}^*\circ P_{E_{r,i}}
\end{equation}
where $a = \ag(\ga) \leqslant N$, the $\xi_i$'s are in $\Ga\setminus \Ga_0$ with $\xi_a = \ga$, the $E_{r,i}$'s are intervals of $\N\cup\{0\}$ with $E_{1,1} < \cdots < E_{1,k_1} < \ra(\xi_1) <  E_{2,1} < \cdots < E_{a,k_a} < \ra (\xi_a)$, the $\eta_{r,i}$'s are in $\Ga$ with $\ra(\eta_{r,i})\in E_{r,i}$, the $\e_{r,i}$'s are in $\{-1,1\}$, $n_{r} > (\#\Ga_{\ra(\xi_{r-1})})^2$ for $r=2,\ldots,a$ and $1\leqslant k_r \leqslant n_r$ for $r=1,\ldots a$.
\end{prp}

\begin{rmk}
Note that if in \eqref{analyse this} we set $f_r = 1/n_r\sum_{i=1}^{k_r}e_{\eta_{r,i}}^*\circ P_{E_{r,i}}$, then $(f_r)_{r=1}^a$ constitutes a very fast growing sequence of $\al$-averages, in the sense of \cite{ABM}.
\end{rmk}

\subsection{The main property of the space $\X$}

\begin{lem}\label{useful lema}
Let \(x_1 < \cdots < x_m \) be blocks of \((d_{\ga_i})_i\) of norm at most one, $E_1<\cdots<E_k$ be intervals of $\N\cup\{0\}$, $(\eta_i)_{i=1}^k$ be a sequence in $\Ga$ with $\ra(\eta_i)\in E_i$,  for $i=1,\ldots,k$, $(\e_i)_{i=1}^k$ be a sequence in $\{-1,1\}$ and let \(n \geq \max \{m^2, k\}\). Then
\begin{equation}
\left|\left(\frac{1}{n} \sum_{i=1}^k \e_ie_{\eta_i}^*\circ P_{E_i}\right)\left(\sum_{j=1}^m x_j\right) \right| \leqslant \frac{4N}{N-2\theta}.
\end{equation}
\end{lem}

\begin{proof}
We shall consider the support and range of vectors with respect to the basis $(d_{\ga_i})_i$. Set $x_i^* = \e_ie_{\eta_i}^*\circ P_{E_i}$ for $i=1,\ldots,k$. Note that $(x_i^*)_{i=1}^k$ is a successive block sequence of $(d_{\ga_i}^*)_i$ and by Remark \ref{bd norms in this particular case}, $\|x_i^*\| \leq 2N/(N-2\theta)$ for $i=1,\ldots,k$. Let \(I_1\) be the set of those \(i \leq k\) such that the support of \(x_i^*\) intersects the range of at least two of the \(x_j\)'s. Then \(\#I_1 \leq m\) and so \(|(1/n \sum_{i \in I_1} x_i^*)(\sum_{j=1}^m x_j) | \leqslant (2N/(N-2\theta))m^2/n \leq 2N/(N-\theta)\). Let \(I_2 = \{i \leq n : \, i \notin I_1\}\). It is clear that \(|(1/n\sum_{i \in I_2} x_i^*)(\sum_{j=1}^m x_j) | \leqslant 2N/(N-2\theta)\) and the proof is complete.
\end{proof}

\begin{prp}\label{asymptotic c0 basis}
Set $K_{N,\theta} = (2N^3 + 4\theta N^2 - 4\theta N)/(N^2 - 3\theta N +2\theta^2)$ and let \(u_1 < \cdots < u_m \) be blocks of \((d_{\ga_i})_i\) of norm at most equal to one so that if we consider the support of the vectors $u_i$ with respect to the basis \((d_{\ga_i})_i\) then, \(m \leq \min \supp u_1\). Then \(\|\sum_{i=1}^m u_i \| \leqslant K_{N,\theta}\).
\end{prp}

\begin{proof}
Set $u = \sum_{i=1}^m u_i$. We use induction on $\ra(\gamma)$ to show that for every $\ga\in\Ga$ and every interval $E$ of $\N$, $|e_\ga^*\circ P_E(u)| \leqslant K_{N,\theta}$. This assertion is easy when \(\ra(\ga) = 0\). Assume that $q\in\N\cup\{0\}$ is such that the assertion holds for every $\ga\in\Ga_q$ and $E$ interval of $\N$ and let $\ga\in\Ga_{q+1}\setminus\Ga_q$ and $E$ be an interval of $\N$. Applying Proposition \ref{analysis}, write
\begin{equation*}
e_\ga^* = \sum_{r = 1}^a d_{\xi_r}^* + \frac{\theta}{N}\sum_{r=1}^a \frac{1}{n_r} \sum_{i=1}^{k_r}\e_{r,i}e_{\eta_{r,i}}^*\circ P_{E_{r,i}}
\end{equation*}
so that the conclusion of that proposition is satisfied. For $r = 1,\ldots,a$ define $y_r^* = (1/n_r)\sum_{i=1}^{k_r}\e_{r,i}e_{\eta_{r,i}}^*\circ P_{E_{r,i}\cap E}$ and $G = \{r:\;\ra(\xi_r)\in E\}$. Observe that $e_\ga^*\circ P_E = \sum_{r\in G}d_{\xi_r}^* + (\theta/N)\sum_{r=1}^ay_r^*$.

If $\xi_r = \ga_{i_r}$ note that $i_1<\cdots<i_a$. Considering the support of $u$ with respect to the basis $(d_{\ga_i})_i$, set $r_0 = \min\{r :\;i_r\geqslant \min\supp u\}$. The inductive assumption yields that $|y_{r_0}^*(u)| \leqslant K_{N,\theta}$, while the growth condition on the $n_r$'s implies that $n_r > m^2$ for $r > r_0$. Lemma \ref{useful lema} yields
\begin{equation}\label{averages part}
\frac{\theta}{N}\sum_{r=1}^a|y_r^*(u)| \leqslant \frac{\theta}{N}\left(K_{N,\theta} + (N-1)\frac{4N}{N-2\theta}\right)
\end{equation}
while Remark \ref{bd norms in this particular case} implies
\begin{equation}\label{bd part}
\sum_{r\in G} |d_{\xi_r}^*(u)| \leqslant N\frac{2N}{N-2\theta}.
\end{equation}
Some calculations combining \eqref{averages part} and \eqref{bd part} yield the desired result.
\end{proof}

\begin{prp}\label{as c0 ell1 predual} The space $\X$ is a $\mathcal{L}_\infty$-space with a Schauder basis $(d_{\ga_i})_i$ satisfying the following properties.
\begin{itemize}

\item[(i)] The basis $(d_{\ga_i})_i$ is shrinking, in particular $\X^*$ is isomorphic to $\ell_1$.

\item[(ii)] The space $\X$ is asymptotic $c_0$ with respect to the basis $(d_{\ga_i})_i$.

\item[(iii)] The space $\X$ does not contain an isomorphic copy of $c_0$.

\end{itemize}
\end{prp}

\begin{proof}
The fact that \(\X\) is asymptotic \(c_0\) with respect to \((d_{\ga_i})_i\) follows directly from Proposition \ref{asymptotic c0 basis}. We obtain in particular that \((d_{\ga_i})_i\) is shrinking and hence, by Proposition \ref{facts about the dual}, the dual of $\X$ is isomorphic to $\ell_1$. To show that \(\X\) does not contain an isomorphic copy of $c_0$, let us consider the FDD $(M_q)_{q=0}^\infty$ as it was defined in Proposition \ref{it is an fdd}. Let $(x_k)_k$ be a sequence of skipped block vectors, with respect to the FDD $(M_q)_{q=0}^\infty$, all of which have norm at least equal to one. Then for every $\e>0$ there exist a finite subset $I_1$ of $\N$ and $\ga\in\Ga$ so that $e_\ga^*(\sum_{k\in I_1}x_k) \geqslant \theta - \e$. It follows from this that for all \(n \in \mathbb{N}\) we can find a $\ga\in\Ga$ and find a finite subset \(J\) of \(\mathbb{N}\) so that
\(e_\ga^*(\sum_{i \in J} x_k ) \geqslant (\theta-\e)^n\). Therefore, \(c_0\) is indeed not isomorphic to a subspace of \(\X\).
\end{proof}

\begin{rmk}
In \cite{A} D. E. Alspach proved that the $\mathcal{L}_\infty$-space with separable dual defined in \cite{BD} has Szlenk index $\omega$. We note that the space $\X$ has Szlenk index $\omega$ as well. Indeed, if this were not the case then by \cite[Theorem 1]{AB} we would conclude that $\X$ has a quotient isomorphic to $C(\omega^\omega)$ which would imply that $\X$ admits an $\ell_1$ spreading model.
\end{rmk}

\begin{rmk}\label{reflexively and unconditionally saturated}
Every skipped block sequence, with respect to the FDD $(M_q)_{q=1}^\infty$, is boundedly complete and hence, the space $\X$ is reflexively saturated and also every block sequence in $\X$ contains a further block sequence which is unconditional. We also note that a similar method can be used to construct a reflexive asymptotic $c_0$ space with an unconditional basis. This space is related to Tsirelson's original Banach space \cite{T}.
\end{rmk}

\end{document}